\documentclass[12pt, reqno]{amsart}
\usepackage{amsmath,amssymb,amsthm}
\allowdisplaybreaks
\usepackage[colorlinks=true,linkcolor=blue,citecolor=blue,pdfpagelabels=false]{hyperref}
\newtheorem{thm}{Theorem}[section]
\newtheorem{lem}[thm]{Lemma}
\newtheorem{prop}[thm]{Proposition}
\newtheorem{cor}[thm]{Corollary}
\newtheorem{defn}[thm]{Definition}
\newcommand{\thmref}[1]{Theorem~\ref{#1}}
\newcommand{\lemref}[1]{Lemma~\ref{#1}}
\newcommand{\rmkref}[1]{Remark~\ref{#1}}
\newcommand{\propref}[1]{Proposition~\ref{#1}}
\newcommand{\corref}[1]{Corollary~\ref{#1}}

\theoremstyle{remark}
\newtheorem{rmk}{Remark}[section]

\renewcommand{\geq}{\geqslant}
\renewcommand{\leq}{\leqslant}

\begin{document}

\title{Special Values of Shifted Dirichlet Series from an Adjoint Map on Almost Holomorphic Modular Forms}
\author{Raveena Ganash}
\address
{School of Mathematical Sciences, National Institute of Science Education and Research, An OCC of  Homi Bhabha National Institute, Bhubaneswar,  
Via: Jatni, Khurda, Odisha - 752 050, India.}
\email{raveena.ganash@niser.ac.in}

\subjclass[2010]{Primary 11F25; Secondary 11A25, 11F11, 11F30}
\keywords{modular forms of integral weight, Dirichlet series, differential operators}
\date{\today}

\begin{abstract}
We find the adjoint of a map on the space of almost holomorphic modular forms, analogous to the generalized Ramanujan-Serre derivative on the space of quasi-modular forms. As applications, we explicitly find special values of certain shifted Dirichlet series associated to quasi-modular forms and prove a rationality result for the shifted Dirichlet series.
\end{abstract}

\maketitle

\section{Introduction}
Let $M_k$ be the space of modular forms of weight $k$ on  $\Gamma=SL_2(\mathbb{Z})$ and $S_k$ be the subspace of cusp forms for an even integer $k\geq 4.$ For $f=\sum_{m \geq 0} a_m q^m \in M_k,$  $g =\sum_{m\geq 0}b_m q^m\in M_l$ and any positive integer $n,$ we define the shifted Dirichlet series by 
\begin{equation}
  L_{f,g,n}(w)=\sum_{m \geq 0} \dfrac{a(n+m)\overline{b(n)}}{(n+m)^{w}}  
\end{equation}
Using the adjoint of the product map, Kohnen \cite{kohnen} constructed cusp forms which involve shifted Dirichlet series as Fourier coefficients. More precisely,
\begin{thm} \cite{kohnen}
Let $k$ and $l$ be two even integers with $k >l+2$ and let $f,$ $g$ be two cusp forms of weight $k$ and $l,$ respectively. Then, the series 
\begin{equation}
  \sum_{n \geq 1} n^{k-l-1} L_{f,g,n}(k-1)q^n 
\end{equation}
is a cusp form of weight $k-l.$
\end{thm}

Further, Herrero extended the construction of Kohnen  by considering the adjoint of the linear maps defined by the Rankin–Cohen brackets, rather than the product map.

\bigskip 
For an even integer $k\geq 2$ and $s\in \mathbb{Z}_{\geq 0},$ let $\widetilde{M}_k^{\leq s}(\Gamma)$ be the space of quasi-modular forms of weight $k,$ depth $\leq s$ (these spaces can be defined for any congruence subgroup of $\Gamma$).
The Eisenstein series
\begin{equation*}
    E_2(z)=1-24\sum_{n\geq1}\sigma(n)e^{2 \pi in z}
\end{equation*}
where $\sigma(n)=\sum_{d|n}d,$ is a quasi-modular form of weight $2$ and depth $1.$ 
Wang \cite{wang} generalized the work of Kohnen to the space of quasi-modular forms and prove the following result.
\begin{thm}
    Let \(k\), \(l\) and \(s\) be three positive integers. Let \(f\) and \(g\) be two nonzero elements of \(S_{k}\) and \(M_{l}\), respectively. Suppose that either 
\begin{itemize}
\item[(a)] \(g\) is a cusp form and \(s \leq (k - l) / 2 - 2\), or 
\item[(b)] \(g\) is not a cusp form and \(s\leq k / 2 - l - 1\).
\end{itemize}
Then, the series \(\sum_{n \geq 1}n^{k-l-s-1} L_{f,g,n}(k-s-1) q^n\) is a quasi-modular form of weight \(k - l\), depth \(\leq s\).
\end{thm}
\bigskip
 
For $f \in M_k,$  $Df:=\dfrac{1}{2 \pi i}\dfrac{df}{dz} \in M_{k+2}^{\leq 1}$  but $\nu_k(f):=Df-\frac{k}{12}E_2f \in M_{k+2}.$ The map $\nu_k:M_k \rightarrow M_{k+2}$ is called the Ramanujan-Serre derivative.

In \cite{ganash-sahu}, we give a generalization of the Ramanujan-Serre derivative to the space of quasi-modular forms. The linear map $\delta_{k,s}:\widetilde{M}_k^{\leq s} \rightarrow \widetilde{M}_k^{\leq s-1}$ is given by 
$\delta_{k,s}(f)=f-\left(\dfrac{2 \pi i}{12}\right)^sf_sE_2^s,$ where  $f_s$ is the $s$-th quasi-component of $f.$

\bigskip

Let $\widehat{M}_k^{\leq s}$ denote the space of almost holomorphic modular forms of weight $k$ and depth $\leq s$  on $\Gamma.$ We define a map $\psi_{k,s}:\widehat{M}_k^{\leq s} \rightarrow \widehat{M}_k^{\leq s-1}$ which is analogous to $\delta_{k,s}.$ Then, we find the adjoint of the map $\psi_{k,s}$ with respect to the Petersson scalar product. As a consequence, we construct new cusp forms as well as quasi-cusp forms with Fourier coefficients involving terms of certain special values of shifted Dirichlet series associated to powers of $E_2$ and quasi-cusp forms. 
In \cite{ak}, Kumar computed the adjoint of $\nu_k,$ and  constructed cusp forms which involve special values of shifted Dirichlet series associated to a cusp form and $E_2.$ Kumar's result follows as a consequence of our main theorem.

\section{Preliminaries:}
For any function $f$ on the upper half plane $\mathfrak{H}:=\{z\in \mathbb{C}:Im(z)>0\}$, $k \in \mathbb{Z}$ and $\gamma = \left( \begin{array}{cc}a & b\\ c & d \end{array} \right) \in \mathrm{GL}_{2}^{+}(\mathbb{R})$, the slash operator is defined by
\[
(f|_{k}\gamma)(z):= (\operatorname{det}\gamma)^{k / 2}j(\gamma ,z)^{-k}f(\gamma z),\quad j(\gamma ,z) = cz + d.
\]
We begin with a review of quasi-modular forms. 
\begin{defn}
A quasi-modular form of weight $k$, depth $s$ on $\Gamma$ is a holomorphic function $f$ on $\mathfrak{H}$ with a collection of component functions (called quasi-components) $f_{0} = f,f_{1},\dots ,f_{s}\neq 0$ over $\mathfrak{H}$, such that
\[
(f|_{k}\gamma)(z)=\sum_{0 \leq r \leq s}f_{r}(z)\Big(\frac{c}{cz + d}\Big)^{r}\quad \text{for every }\gamma = \left( \begin{array}{cc}a & b\\ c & d \end{array} \right)\in \Gamma,
\]
where each $f_{r}$ is holomorphic on the upper half plane as well as at the cusps.
\end{defn}
One can define quasi-modular forms for any level $N,$ but in this text we will deal with $N=1$ only.     
An important example of a quasi-modular form is the weight $2$ and depth $1$ Eisenstein series
$
E_{2}(z) = 1 - 24\sum_{n = 1}^{\infty} \sigma(n) e^{2\pi inz}.
$ It transforms under the action of $\Gamma$ as
\[
(E_{2}|_{2}\gamma)(z) = E_{2}(z) + \frac{6}{\pi i} \frac{c}{cz + d},\quad \gamma = \begin{pmatrix} a & b \\ c & d \end{pmatrix} \in \Gamma.
\]

It is a fact by Kaneko and Zagier in \cite{kaneko1995generalized}  that 
  every quasi-modular form on \(\Gamma\) can be written uniquely as a linear combination of derivatives of modular forms and of \(E_2\). More precisely, for all \(k, s \geq 0\) we have
\[
\widetilde{M}_k^{\leq s} =
\begin{cases} 
\displaystyle \bigoplus_{r=0}^s D^r(M_{k-2r}) & \text{if } s < k/2, \\[15pt]
\displaystyle \bigoplus_{r=0}^{k/2-1} D^r(M_{k-2r}) \oplus \mathbb{C} \cdot D^{k/2-1}E_2 & \text{if } s \geq k/2.
\end{cases}
 \]

For more details on quasi-modular forms, we refer to \cite{zagier2008elliptic} or \cite{strom}. 
Let $\widetilde{M}_{k}^{s}$ be the set of quasi-modular forms of weight $k$, depth $s$ and $\widetilde{S}_{k}^{s}$ be the subset of $\widetilde{M}_{k}^{s}$ with each quasi-component $f_{r}$ vanishing at all cusps of $\Gamma$ (we will call it quasi-cusp forms) and $\widetilde{S}_{k}^{\leq s}:= \bigoplus_{r = 0}^{s}\widetilde{S}_{k}^{r}$. Clearly, $\widetilde{S}_{k}^{\leq s}$ is a subspace of $\widetilde{M}_{k}^{\leq s}.$
We have the following structure theorem  for the space $\widetilde{S}_{k}^{\leq s}.$

\begin{prop}\cite{wang}\label{bound}
\[
\widetilde{S}_{k}^{\leq s} = \bigoplus_{r = 0}^{s} S_{k - 2r}E_2^r,
\]
and
\[
\widetilde{S}_{k}^{\leq s} = \bigoplus_{r = 0}^{s} D^{r}(S_{k - 2r}),
\]
for $0 \leq s \leq k / 2 - 1.$ Hence, for any $f \in \widetilde{S}_{k}^{\leq s}$ with Fourier expansion $\sum_{n \geq 1} a(n) q^{n}$, ($q=e^{2 \pi i z}$) we have $a(n) \ll n^{\frac{k - 1}{2} + \varepsilon}$ for any $\varepsilon > 0.$
\end{prop}

We have the following analouge of Poincar\'e series in case of quasi-modular forms. 
\begin{prop}\label{prop:quasi-poincare}\cite{wang}
Let \(k,s\) be two positive integers with $k-2s>2.$
For every positive integer \(n\), define
\[
\widetilde{P}_{k,n}^{s}(z)=\sum_{\gamma\in\Gamma_{\infty}\backslash\Gamma} c_{\gamma}^{s}\,e^{2\pi i n\gamma z}\,j(\gamma,z)^{-k},
\]
where \(c_{\gamma}\) denotes the lower-left entry of \(\gamma.\)
Then \(\widetilde{P}_{k,n}^{s}(z)\) converges absolutely and uniformly on compact subsets of \(\mathcal{H}\), and it belongs to \(\widetilde{S}_{k+s}^{\le s}.\)
It has depth \(s\) if and only if the classical Poincar\'e series \(\widetilde{P}_{k-2s,n}^{0}\) is non-zero.
\end{prop}
Here, $\Gamma_{\infty}=\{ \pm\begin{pmatrix}
1 & n \\
0 & 1
\end{pmatrix}|n \in \mathbb{Z}\}
$ is the group of translations.
We define almost holomorphic modular forms.
\begin{defn}
A function $F$  on $\mathfrak{H}$ is said to be an almost holomorphic modular form of weight $k$ and depth $s$ if  $(F|_{k}\gamma)(z) = F(z)$ for any $\gamma \in \Gamma$, and it is of the form
\[
F(z) = \sum_{0 \leq r \leq s} F_{r}(z) (2iy)^{-r},
\]
where each $F_{r}(z)$ is holomorphic on $\mathfrak{H}$ as well as at every cusp of $\Gamma$ and $F_s$ is not the zero function.
\end{defn}

Let  $\widehat{M}_{k}^{s}$ be the set of almost holomorphic modular forms of weight $k,$ depth $s$ and $\widehat{M}_{k}^{\leq s}:= \bigoplus_{r = 0}^{s}\widehat{M}_{k}^{r}$ be the space of almost holomorphic modular forms of depth $\leq s.$ For example, $E_2^*(z)=E_2(z)-\frac{3}{\pi y}$ is an almost holomorphic modular form of weight $2$ and depth $1$.
\bigskip

Almost holomorphic modular forms and quasi-modular forms are strongly interconnected. 
The modularity of $F(z)$ implies that each $F_{r}(z)$ is a quasi-modular form of weight $k - 2r$ and depth $s - r.$ There is a canonical isomorphism between $\widehat{M}_{k}^{s}$ and $\widetilde{M}_{k}^{s}$ which is given by the constant term of $F(z)$:
\[
\delta_0 :\widehat{M}_{k}^{s}\to \widetilde{M}_{k}^{s}, \quad \delta_0 (F(z)):= F_{0}(z).
\]

In this correspondence, it is not hard to see that  $F\in \widehat{M}_k^s$ corresponds to $F_0\in \widetilde{M}_k^s$ with the quasi-component $F_0$,$F_1$,\ldots,$F_s$. For instance, $E_2^*$ maps to the quasi-modular form $E_2$ under this correspondence. We refer to \cite{strom} for more details about the map $\delta_0.$ Denote by $\widehat{S}_{k}^{s}$ the set of almost holomorphic modular forms corresponding to $\widetilde{S}_{k}^{s}$, and let $\widehat{S}_{k}^{\leq s}:= \bigoplus_{r = 0}^{s}\widehat{S}_{k}^{r}$. 
The following is the analogue of Poincar\'e series in the space of almost holomorphic modular forms.
\begin{prop}\cite{wang}
Let $k$ and $s$ be two positive integers with $k - 2s > 2.$ For every positive integer $n$, define
\[
\widehat{P}_{k,n}^{s}(z) = \sum_{\gamma \in \Gamma_{\infty}\backslash \Gamma}\mathrm{Im}(\gamma z)^{-s}e^{2\pi i n\gamma z}j(\gamma ,z)^{-k}.
\]
The series defined above is absolutely and uniformly convergent on any compact subset of $\mathfrak{H}$, and is an element of $\widehat{S}_{k}^{\leq s}.$ Under the  canonical isomorphism $\delta_0$, it corresponds to the quasi-modular form $(- 2i)^{s}\widetilde{P}_{k - s,n}^{s}(z).$
\end{prop}
The space $\widehat{S}_{k}^{\leq s}$ is a finite dimensional Hilbert space with respect to the Petersson inner product
\begin{equation*}
    <f,g>:=\int_{\mathbb{F}} f(z)\overline{g(z)}y^k\dfrac{dxdy}{y^2}
\end{equation*}
for $f,g\in \widehat{S}_{k}^{\leq s},$ ${\mathbb{F}}$ is a fundamental domain for the action of $\Gamma$ on $\mathfrak{H}.$
\begin{rmk}
One can check (or see \cite{wang}) that 
\begin{equation*}
(2i)^{-s}\widehat{P}^s_{k,n}(z)=\sum_{t=0}^s(-1)^t\binom{s}{t}\widetilde{P}^t_{k+t-2s,n}(z)(2iy)^{t-s}
\end{equation*}
which implies that the $s$-th quasi-component of $\widehat{P}^s_{k,n}$ is  $(2i)^s\widetilde{P}^0_{k-2s,n},$ where ${P}^0_{k-2s,n}$ is the classical Poincar\'e series
\end{rmk}
\begin{rmk}\label{fou-coe}
For  $F\in \widehat{S}_{k}^{\leq s}$ such that
$F(z)=\sum_{n \geq 1} \left(\sum_{t=0}^sa_t(n)y^{-t}\right)q^n $
and $0 \leq d \leq s$, we have
\begin{equation}
<F,\widehat{P}^d_{k,n}>=\sum_{t=0}^s\dfrac{\Gamma(k-d-t-1)a_t(n)}{(4\pi n)^{k-d-t-1}}   
\end{equation}
\end{rmk}

The following is a result by Ganash and Sahu in \cite{ganash-sahu}.
\begin{thm}\label{serre-quasi}\cite{ganash-sahu}.
Let $f$ be a quasi-modular form in $\widetilde{M}_{k}^{\leq s}$ with the $s$-th quasi-component $f_s$. Then $f-\left(\frac{2 \pi i}{12}\right)^s E_2^s f_s$ is a quasi-modular form in $\widetilde{M}_{k}^{\leq s-1}$. Also, the map $\delta_{k,s}:\widetilde{M}_{k}^{\leq s} \rightarrow \widetilde{M}_{k}^{\leq s-1} $ given by $\delta_{k,s}(f)=f-\left(\frac{2 \pi i}{12}\right)^s E_2^s f_s$   is linear.
\end{thm}
The aim of the paper is to find the analogue of the above result to the space of almost holomorphic modular forms and then compute its adjoint map with respect to the Petersson scalar product.

\section{Results}
\begin{lem}\label{ganash-sahu-a}
Let $F(z)=\sum_{0\leq r \leq s}F_r(z)(2 iy)^{-r}$ be an almost holomorphic form in $\widehat{M}_{k}^{\leq s}$. Then, $F-\left(\frac{2 \pi i}{12}\right)^s (E_2^*)^s F_s\in \widehat{M}_{k}^{\leq s-1} $ and the map $\psi_{k,s}:\widehat{M}_{k}^{\leq s} \rightarrow \widehat{M}_{k}^{\leq s-1}$ given by $\psi_{k,s}(F)=F-\left(\frac{2 \pi i}{12}\right)^s (E_2^*)^s F_s$   is linear.
\end{lem}
For convenience, we write $\psi$ for $\psi_{k,s},$ for given $k$ and $s.$
For a quasi-cusp form $f\in \tilde{S}_{k}^{\leq s}$  with Fourier expansion $f=\sum_{n \geq 1}a(n)q^n$, we define the following shifted Dirichlet series
\begin{equation}\label{lseries}
    L_{E_2^j,f,n}(w):=\sum_{m \geq j} \dfrac{\lambda_j(m)a(n+m)}{(n+m)^w},
\end{equation}
where $j \in \{1,\ldots,s\},$  $\lambda_j(m)=\sum_{n_1+\ldots+n_j=m}\sigma(n_1)\ldots\sigma(n_j).$   The series converges absolutely for $Re(w)>2s+\frac{k-1}{2}+\epsilon\quad\text{for any}\quad\epsilon>0.$ One can check the convergence using the bound given in \propref{bound} and the fact  $\lambda_j(m)=O(m^{2j-1+\epsilon})$ for any $\epsilon>0.$
The adjoint of $\psi:\widehat{S}_{k}^{\leq s} \rightarrow \widehat{S}_{k}^{\leq s-1}$ is the map $\psi^{*}:\widehat{S}_{k}^{\leq s-1}\rightarrow \widehat{S}_{k}^{\leq s}$ satisfying $<\psi(F_1),F_2>=<F_1,\psi^*(F_2)>$ for all $F_1 \in \widehat{S}_{k}^{\leq s}$ and $F_2 \in \widehat{S}_{k}^{\leq s-1}.$ We establish the adjoint of $\psi$ in the following result.\\

\begin{thm}\label{mainthm}
Let $k,s\in \mathbb{N}$, $k$ even and
$G(z)=\sum_{t^{'}=0}^{s-1}g_{t'}(z)y^{-t'}=\sum_{n \geq 1}\left(\sum_{t^{'}=0}^{s-1}b_{t^{'}}(n)y^{-t'}\right)q^n \\\in \widehat{S}_k^{\leq s-1}$ such that $k \geq 8s-1.$ 
The adjoint of $\psi$ is given by
\begin{equation*}
    \psi^*(G)=\sum_{n \geq 1}(\sum_{t=0}^{s}a_{t}(n)y^{-t})q^n
\end{equation*}
where $a_t(n)$ are given by $X=A^{-1}Y,$ with 
$X=[a_0(n),\ldots,a_s(n)]^T,$ 
\begin{equation}\label{matrix}
    A=\begin{pmatrix}
 1&\dfrac{(4 \pi n)}{(k-2)}&\dfrac{(4 \pi n)^2}{(k-2)(k-3)} &\ldots&\dfrac{(4 \pi n)^s}{(k-2)\ldots(k-s-1)} \\
 1&\dfrac{(4 \pi n)}{(k-3)}&\dfrac{(4 \pi n)^2}{(k-3)(k-4)} &\ldots&\dfrac{(4 \pi n)^s}{(k-3)\ldots(k-s-2)} \\
\vdots&\vdots&\vdots&\ldots&\vdots \\
 1&\dfrac{(4 \pi n)}{(k-s-2)}&\dfrac{(4 \pi n)^2}{(k-s-2)(k-s-3)} &\ldots&\dfrac{(4 \pi n)^s}{(k-s-2)\ldots(k-2s-1)} \\
\end{pmatrix},
\end{equation}
 $Y=[c_0(n),\ldots,c_s(n)]^t$ where
\[
c_d(n) =b_0(n) +\dfrac{4 \pi n}{(k-d-2)}b_1(n)+...+\dfrac{(4 \pi n)^{s-1}}{(k-d-2)\ldots(k-d-s)}b_{s-1}(n)
\]
for $0 \leq d \leq s-1$ and 
\begin{align*}
c_s(n)=&b_0(n) +\dfrac{4 \pi n}{(k-s-2)}b_1(n)+...+\dfrac{(4 \pi n)^{s-1}}{(k-s-2)\ldots(k-2s)}b_{s-1}(n)-\\&
\dfrac{(4 \pi n)^{k-s-1}}{\Gamma(k-s-1)} \Big(\sum_{r=0}^s\binom{s}{r}\left(\dfrac{-3}{\pi }\right)^{r-s}\sum_{t'=0}^{s-1} \dfrac{b_{t'}(n)}{(4 \pi n)^{k-t'-r-1}}\Gamma(k-t'-r-1)+\\
&\sum_{j=1}^s\sum_{l=0}^{s-j} \binom{s}{j} \binom{s-j}{l}\left(\dfrac{-3}{\pi }\right)^{l-s}(-24)^j  \sum_{t'=0}^{s-1}\dfrac{L_{E_2^j,g_t',n}(k-l-t'-1)}{(4 \pi)^{k-l-t'-1}}\\&\Gamma(k-l-t'-1)\Big).
\end{align*}
\end{thm}

\begin{rmk}
One can check that the matrix $A$ is invertible by computing the determinant of the matrix (see \cite{wang}).  
\end{rmk}
We have the following corollary:
\begin{cor}\label{cor}
    Let $G=\sum_{n\geq 1}b(n)q^{n}\in S_{k+2},$ then $\psi^*(G)=\sum_{n \geq 1}\left(a_0(n)+\dfrac{a_1(n)}{y}\right)q^n\in \widehat{S}_{k+2}^{\leq 1},$ where  
\begin{equation}\label{a0}
a_0(n)=-b(n)\left(k+\dfrac{k(k-1)}{12n}\right)+2n^kk(k-1)L_{E_2,G,n}(k+1)
\end{equation}
and 
\begin{equation}\label{a1}
a_1(n)=-\dfrac{k(k-1)}{4 \pi n}\left(b(n)\left(\dfrac{k}{12n}-1\right)-2n^kkL_{E_2,G,n}(k+1)\right).
\end{equation}
Furthermore, $\sum_{n \geq 1} a_0(n)q^n$ is a  quasi-cusp form of weight $k+2$, depth $\leq 1$ and $\sum_{n \geq 1} a_1(n)q^n$ is a  cusp form of weight $k.$
\end{cor}

\smallskip
\begin{rmk}\label{adj-serre}
Following the notations of the above corollary, the map  $\mu: S_{k+2} \rightarrow S_{k}$ given by $\mu(G)=\sum_{n \geq 1}a_1(n)q^n$ is the adjoint map (given in \cite{ak}) of the Ramanujan-Serre derivative $\nu_{k}:S_k \rightarrow S_{k+2}$    given by $\nu_k(f)=Df-\frac{k}{12}E_2f$ (upto scalar multiplication).
\end{rmk}

 \begin{cor}\label{rational}
 If $g=\sum_{l\geq 1}b(l)q^l \in ker(\mu)$ is a cusp form of weight $k+2$ such that $b(l)\in \mathbb{Q}$ for all $l\leq dim(S_{k+2})+1.$ Then, $L_{E_2,g,n}(k+1) \in \mathbb{Q}$ for all $n\in \mathbb{N}.$ In fact,
\begin{equation}\label{LS}
 L_{E_2,g,n}(k+1)=\dfrac{(k-12n)}{24n^{k+1}k}  b(n).
\end{equation}
 \end{cor}
\begin{proof}
Since $b(l)\in \mathbb{Q}$ for all $l\leq dim(S_{k+2})+1$, due to the Sturm bound we get that $b(l)\in \mathbb{Q}$ for all $l \in \mathbb{N}$. As
$g=\sum_{l\geq 1}b(l)q^l \in ker(\mu)$, due to \rmkref{adj-serre} and \corref{cor} we get that $\sum_na_1(n)q^n$ is the zero function.  Substituting the expression for $a_1(n)$ to zero, we get the proposed expression \eqref{LS}.
\end{proof} 
\begin{rmk}
Note the fact that if $V, W$ are finite dimensional vector spaces and $T:V \rightarrow W$ is a linear map. Then, $T^*:W \rightarrow V$ is the adjoint map such that $<T(v),w>=<v,T^*(w)>$ for all $v \in V$ and $w \in W$. We know, $range(T)^{\perp}=ker(T^*).$ So, this along with the fact that the map $\nu_k$ is not surjective for infinitely many weights $k,$ ensures that there are ample choices for such a $g.$
\end{rmk}

\section{Proofs}
We state the following lemma which we  need in the proof of \thmref{mainthm}.
\begin{lem}\label{tech}
For  $k \geq 4$ even, $s \in \mathbb{Z}_{\geq 0}$ be such that $k \geq 8s-1$ and $G(z)=\sum_{u \geq 1}(\sum_{t'=0}^{s-1}b_{t'}(u))q^u \in \widehat{S}_k^{\leq s-1},$ the following series
\begin{align}\label{con}
\sum_{\gamma \in \Gamma_{\infty} \backslash \Gamma}\int_{\mathbb{F}}  |G(z)q^n|_{k-2s}\gamma (E_2^*)^s|y^k \dfrac{dxdy}{y^2}
\end{align}
 converges, where $\mathbb{F}$ is a fundamental domain for the action of $\Gamma$ on $\mathfrak{H}$.
 \end{lem}
\begin{proof}
Making the substitution $z$ to $M^{-1}z$, we get
\begin{align*}
\sum_{M \in \Gamma_{\infty} \backslash \Gamma}\int_{\mathbb{F}}  |G(z)q^n|_{k-2s}M (E_2^*)^s|y^k \dfrac{dxdy}{y^2}=\int_{\Gamma_{\infty}\backslash \mathbb{F}}    |G(z)q^n(E_2^*)^s|y^k \dfrac{dxdy}{y^2},
\end{align*} where $\Gamma_{\infty}\backslash\mathbb{F}=\{(x,y)|0\leq x \leq 1,0<y\}$ is the fundamental domain for the action of $\Gamma_{\infty}$ on $\mathfrak{H}$.
Note 
\begin{align}\label{EFC}
E_2^*(z)^s=&\left(1-24\sum_{m \geq 1}\sigma(m)q^m-\dfrac{3}{\pi y}\right)^s \notag \\
=&\sum_{r=0}^s\binom{s}{r} \left(\dfrac{-3}{\pi y}\right)^r+\sum_{j=1}^s\sum_{l=0}^{s-j} \binom{s}{j} \binom{s-j}{l}\left(\dfrac{-3}{\pi y}\right)^l(-24)^j \sum_{m \geq j} \lambda_j(m)q^m
\end{align}
where $\lambda_j(m)=\sum_{n_1+\ldots+n_j=m}\sigma(n_1)\ldots\sigma(n_j)$. Also, $\lambda_j(m)=O(m^{2j-1+\epsilon})$ and a bound for the Fourier coefficients of $G$ is $b_{t'}(u)=O(u^{\frac{k-1}{2}+\epsilon})$ for all $t'$, due to \propref{bound}.
Substituting the Fourier expansion of $G$, the above expression for $E_2^*$, and using triangular inequality, we get
\begin{align*}
\int_{\Gamma_{\infty}\backslash \mathbb{F}}    |G(z)q^n(E_2^*)^s|y^k \dfrac{dxdy}{y^2} \leq &  \sum_{r=0}^s \binom{s}{r} \left( \dfrac{3}{\pi}\right)^r \int_{\Gamma_{\infty}\backslash \mathbb{F}}   |G(z)q^n|y^{k-r-2}dxdy\\&+
\sum_{j=1}^s\sum_{l=0}^{s-j} \binom{s}{j} \binom{s-j}{l} \left(\dfrac{3}{\pi} \right)^l(24)^j \\&\int_{\Gamma_{\infty}\backslash \mathbb{F}}  |G(z)q^n \sum_{m \geq j} \lambda_j(m)q^m| y^{k-2-l} dxdy.
\end{align*}
Now we show that the integrals on the right hand side are finite. The first integral 
\begin{align*}
\int_{\Gamma_{\infty}\backslash \mathbb{F}}   |G(z)q^n|y^{k-r-2}dxdy \ll &   \int_{\Gamma_{\infty}\backslash \mathbb{F}} \sum_{u \geq 1} u^{\frac{k-1}{2}+\epsilon}(\sum_{t'=0}^{s-1}y^{-t'})e^{-2 \pi y(u+n)}y^{k-r-2} dxdy\\
&=\sum_{u \geq 1} \sum_{t'=0}^{s-1}u^{\frac{k-1}{2}+\epsilon} \int_{y=0}^{\infty} e^{-2 \pi y(u+n)} y^{k-r-t'-2} dy
\\&=\sum_{u \geq 1} u^{\frac{k-1}{2}+\epsilon} \sum_{t'=0}^{s-1}\dfrac{\Gamma(k-r-t'-1)}{(2 \pi (u+n))^{k-r-t'-1}} < \infty
\end{align*} for $k\geq 4s
+2$. In the second integral, substituting the Fourier expansion of $G$,we have
\begin{align*}
 \int_{\Gamma_{\infty}\backslash \mathbb{F}}  |G(z)q^n \sum_{m \geq j} \lambda_j(m)q^m| y^{k-2-l} dxdy&=\int_{\Gamma_{\infty}\backslash \mathbb{F}} |\sum_{m,u}\sum_{t'=0}^{s-1}\lambda_j(m)b_{t'}(u)q^{n+u+m}y^{k-l-2-t'}| dxdy   
 \\& \ll \sum_{m,u}\sum_{t'=0}^{s-1} m^{2j-1+\epsilon} u^{\frac{k-1}{2}+\epsilon}\int_{y=0}^{\infty} e^{-2 \pi y(n+u+m)}y^{k-l-2-t'} dy
\end{align*}
which is again finite for $k \geq 8s-1$.
\end{proof}
Now, we prove \thmref{ganash-sahu-a}.
\begin{proof}
It is enough to show that the coefficient of $(\dfrac{1}{2iy})^0$ in  $F-\left(\frac{2 \pi i}{12}\right)^s (E_2^*)^s F_s$ is in $\widetilde{M}_k^{\leq s-1}$.
Consider 
\begin{equation*}
F-\left(\frac{2 \pi i}{12}\right)^s (E_2^*)^s F_s=\sum_{0\leq r \leq s}F_r(z)(2 iy)^{-r}-\left(\frac{2 \pi i}{12}\right)^s\left(E_2-\frac{3}{\pi y}\right)^sF_s.  
\end{equation*}
So, the constant term is  $F_0-\left(\dfrac{2 \pi i}{12}\right)^sE_2^sF_s$ which is in $\widetilde{M}_k^{\leq s-1},$ due to \thmref{serre-quasi}.
\end{proof}

Following is the proof of \thmref{mainthm}.
\begin{proof} 
 The adjoint map $\psi^*:\widehat{S}_{k}^{\leq s-1} \rightarrow \widehat{S}_{k}^{\leq s}$ satisfy $<\psi(F),G>=<F,\psi^*(G)>$ for all $F\in \widehat{S}_{k}^{\leq s} $ and $G \in \widehat{S}_{k}^{\leq s-1}.$
For any arbitrary $G(z)=\sum_{u \geq 1}\left(\sum_{t^{'}=0}^{s-1}b_{t^{'}}(u)y^{-t^{'}}\right)q^u \in \widehat{S}_{k}^{\leq s-1} $, say 
$\psi^*(G)=\sum_{m \geq 1}(\sum_{t=0}^{s}a_t(m)y^{-t})q^m.$ We calculate $a_t(m)$ for all $t \in \{0,1,\ldots,s\}$ and for all $m \in \mathbb{N}.$ For any $0\leq d\leq s$, we have (due to \rmkref{fou-coe})
\begin{equation}\label{1}
<\psi^*(G),\widehat{P}_{k,n}^d>=\sum_{t=0}^{s}\dfrac{\Gamma(k-d-t-1)a_t(n)}{(4\pi n)^{k-d-t-1}}.  
\end{equation}
Also, $<\psi^*(G),\widehat{P}_{k,n}^d>=<G,\psi(\widehat{P}_{k,n}^d)>$. 
For $d\in \{0,\ldots,s-1\}$, we have
\begin{align}\label{2}
<G,\psi(\widehat{P}_{k,n}^d)>=<G,\widehat{P}_{k,n}^d>
=\sum_{t^{'}=0}^{s-1}\dfrac{\Gamma(k-d-t^{'}-1)b_{t^{'}}(n)}{(4\pi n)^{k-d-t^{'}-1}},
\end{align} 
and for $d=s$, we have

\begin{align}\label{MAE}
<G,\psi(\widehat{P}_{k,n}^s)>=&<G,\widehat{P}_{k,n}^s>-\left(\frac{2 \pi i}{12}\right)^s(2i)^s<G,P_{k-2s,n}(E_2^*)^s>\notag\\
=&\left(\sum_{t^{'}=0}^{s-1}\dfrac{\Gamma(k-s-t^{'}-1)b_{t^{'}}(n)}{(4\pi n)^{k-s-t^{'}-1}}\right)-\left(\frac{2 \pi i}{12}\right)^s(2i)^s<G,P_{k-2s,n}(E_2^*)^s>.
\end{align}  
Consider
\begin{align*}
<G,P_{k-2s,n}(E_2^*)^s>=\int_{\mathbb{F}} G(z)\overline{P_{k-2s,n}(E_2^*)^s} y^k \dfrac{dxdy}{y^2} 
=\int_{\mathbb{F}} G(z) \overline{\sum_{M \in \Gamma_{\infty}\backslash \Gamma}\left(q^n|_{k-2s}M\right)E_2^{*s}} y^k \dfrac{dxdy}{y^2} 
\end{align*}
Substituting $M^{-1}z$ in place of $z$, interchanging the summation and integration (due to \lemref{tech}) and then, using Rankin's unfolding argument, we get
\begin{align*}
<G,P_{k-2s,n}(E_2^*)^s>&=\int_{0}^{1} \int_{0}^{\infty} G(z)\overline{e^{2 \pi inz}(E_2^*)^s} y^k \dfrac{dxdy}{y^2}\\     
\end{align*}
Substituting the Fourier expansion of $G$ and $E_2^*$ from \eqref{EFC}, we get
\begin{align*}
<G,P_{k-2s,n}(E_2^*)^s>&=\int_{0}^{1} \int_{0}^{\infty} G(z)\overline{e^{2 \pi inz}}\Bigl(\sum_{r=0}^s\binom{s}{r} \left(\dfrac{-3}{\pi y}\right)^r\\
&+\sum_{j=1}^s\sum_{l=0}^{s-j} \binom{s}{j} \binom{s-j}{l}\left(\dfrac{-3}{\pi y}\right)^l(-24)^j \sum_{m \geq j} \lambda_j(m)\overline{q^m}\Bigr) y^k \dfrac{dxdy}{y^2}\\
&=\sum_{r=0}^s\binom{s}{r} \left(\dfrac{-3}{\pi }\right)^r\sum_{t'=0}^{s-1} \dfrac{b_{t'}(n)}{(4 \pi n)^{k-t'-r-1}}\Gamma(k-t'-r-1)
\\&+\sum_{j=1}^s\sum_{l=0}^{s-j} \binom{s}{j} \binom{s-j}{l}\left(\dfrac{-3}{\pi }\right)^l(-24)^j \sum_{m \geq j} \lambda_j(m)\sum_{t'=0}^{s-1}\dfrac{b_{t'}(n+m)}{(4 \pi (m+n))^{k-l-t'-1}}
\\&\Gamma(k-l-t'-1).
\end{align*}
 Substituting the above expression for $<G,P_{k-2s,n}(E_2^*)^s>$ in \eqref{MAE}, we have the following set of $s+1$ equations.
\begin{align*}
\sum_{t=0}^{s}\dfrac{\Gamma(k-d-t-1)a_t(n)}{(4\pi n)^{k-d-t-1}} =& \sum_{t'=0}^{s-1}\dfrac{\Gamma(k-d-t'-1)b_{t'}(n)}{(4\pi n)^{k-d-t'-1}}, 
\end{align*}
for $0\leq d \leq s-1$ and for $d=s$,
\begin{align*}
 \sum_{t=0}^{s}\dfrac{\Gamma(k-s-t-1)a_t(n)}{(4\pi n)^{k-s-t-1}} =& \sum_{t'=0}^{s-1}\dfrac{\Gamma(k-s-t'-1)b_{t'}(n)}{(4\pi n)^{k-s-t'-1}} -\left(\frac{2 \pi i}{12}\right)^s(2i)^s\Big(\sum_{r=0}^s\binom{s}{r}\\
 &\left(\dfrac{-3}{\pi }\right)^r\sum_{t'=0}^{s-1} \dfrac{b_{t'}(n)}{(4 \pi n)^{k-t'-r-1}}\Gamma(k-t'-r-1)+\sum_{j=1}^s\sum_{l=0}^{s-j} \\
 &\binom{s}{j} \binom{s-j}{l}\left(\dfrac{-3}{\pi }\right)^l(-24)^j \sum_{m \geq j} \lambda_j(m)\sum_{t'=0}^{s-1}\dfrac{b_{t'}(n+m)}{(4 \pi (m+n))^{k-l-t'-1}}
 \\&\Gamma(k-l-t'-1)\Big).
\end{align*}

Multiplying both sides of the above set of $s+1$ equations by $\dfrac{(4 \pi n)^{k-d-1}}{\Gamma(k-d-1)}$ for $0 \leq d\leq s$ respectively, we get the matrix form $X=A^{-1}Y$ for the above set of equations  with $X=[a_0(n),\ldots,a_s(n)]^t$, 
$Y=[c_0(n),\ldots,c_{s-1}(n),c_s(n)]$ where 
\[
c_d(n) =b_0(n) +\dfrac{4 \pi n}{(k-d-2)}b_1(n)+\ldots+\dfrac{(4 \pi n)^{s-1}}{(k-d-2)\ldots(k-d-s)}b_{s-1}(n)
\]
for $0 \leq d \leq s-1,$
\begin{align*}
c_s(n)=&b_0(n) +\dfrac{4 \pi n}{(k-s-2)}b_1(n)+\ldots+\dfrac{(4 \pi n)^{s-1}}{(k-s-2)\ldots(k-2s)}b_{s-1}(n)-
\dfrac{(4 \pi n)^{k-s-1}}{\Gamma(k-s-1)}\left(\dfrac{- \pi }{3}\right)^s\\& \Big(\sum_{r=0}^s\binom{s}{r}\left(\dfrac{-3}{\pi }\right)^r\sum_{t'=0}^{s-1} \dfrac{b_{t'}(n)}{(4 \pi n)^{k-t'-r-1}}\Gamma(k-t'-r-1)+\sum_{j=1}^s\sum_{l=0}^{s-j} \binom{s}{j} \binom{s-j}{l}\left(\dfrac{-3}{\pi }\right)^l\\&(-24)^j \sum_{m \geq j} \sum_{t'=0}^{s-1}\dfrac{b_{t'}(n+m)\lambda_j(m)}{(4 \pi (m+n))^{k-l-t'-1}}
 \Gamma(k-l-t'-1)\Big)
\end{align*}
and the matrix $A$ is as given in \eqref{matrix}.
\end{proof}

\section{Applications}
\subsection{Formula for Ramanujan tau function $\tau(n)$}
  Applying   \corref{cor}  to $G(z)=\Delta(z)=\sum_{n\geq1}\tau(n)q^{n}$, we have $\psi^*(G)=\sum_{n \geq 1}\left(a_0(n)+\dfrac{a_1(n)}{y}\right)q^n$, where $\sum_n a_0(n)q^n \in \widetilde{S}_{12}^{\leq 1}=S_{12}$ (which implies, $a_0(n)=\lambda \tau(n)$ for some constant $\lambda$) and $\sum_n a_1(n)q^n \in S_{10}=\{0\}.$ Substituting $b(n)=\tau(n),$  $a_0(n)=\lambda\tau(n)$ in \eqref{a0} and finding the value of $\lambda$ by comparing the Fourier coefficients, we get
    \begin{align}\label{new}
\tau(n)=\dfrac{360 n^{11}}{(-15n+15+360nL_{E_2,\Delta,1}(11))}L_{E_2,\Delta,n}(11)     .
    \end{align}
Now, substituting $a_1(n)=0$ and $b(n)=\tau(n)$ in \eqref{a1}, we get
 \begin{align}\label{akeq}
     \tau(n)=\dfrac{-120n^{11}}{6n-5}L_{E_2,\Delta,n}(11).
 \end{align}
 Further, from \eqref{new} and \eqref{akeq}, we  get
 \begin{equation}\label{akeqn}
L_{E_2,\Delta,1}(11)=-\dfrac{1}{120}     
 \end{equation}
 We note that the formula \eqref{akeq} and \eqref{akeqn} are established in \cite{ak}. However, \eqref{new} is new, as far as the best of our knowledge.
\subsection{A relation between a Convolution sum and special value of a shifted Dirichlet series.} Consider $k=16,$ $s=1$ and $G=E_4\Delta=\sum_{n\geq1}b(n)q^n$. Since, $\widetilde{S}_{16}^{\leq1}=S_{16},$ substituting $a_0(n)=\lambda(\tau(n) + 240\sum_{j=1}^{n-1} \sigma_3(j)\,\tau(n-j)
)$ in \eqref{a0} and finding the value of $\lambda$ by comparing the Fourier coefficients, we get
\begin{equation}
    b(n)=\dfrac{312n^{15}L_{E_2,E_4\Delta,n}(15)}{312nL_{E_2,E_4\Delta,1}(15)+37n+13}.
\end{equation} 
Again, since $S_{14}=\{0\},$ from \eqref{a1}, we get
\begin{equation}\label{con1}
    b(n)=\dfrac{168n^{15}L_{E_2,E_4\Delta,n}(15)}{7-6n}
\end{equation}
where $b(n)=\tau(n)+240\sum_{j=1}^{n-1}\sigma_3(j)\tau(n-j).$ Note that, \eqref{con1} can also be derived as an application of \corref{rational}.

\bigskip

\subsection{Rationality of a linear combination of shifted Dirichlet series at some special values:}
 Taking $k=18$ and $s=2$, since $dim(\widehat{S}_{18}^{\leq 2})=2$ and $S_{14}=\{0\}$, consider $G=E_2^*E_4\Delta$, substituting all these values in  \thmref{mainthm}, we get that $a_2(n)=0\quad\text{for all} \quad n$, which simplifies to
 \begin{equation}\label{l}
28(6n^2 - 15n + 10)\, A(n) - n(24n^2 - 56n + 35)\, F(n) \;=\; 28\, n^{17}\, S
 \end{equation}
 where
\begin{equation*}
A(n)=\tau(n) + 240\sum_{m=1}^{n-1}\sigma_3(m)\tau(n-m) - 24\sum_{m=1}^{n-1}\sigma_1(m)\tau(n-m)- 5760\sum_{i=1}^{n-2}\sum_{j=1}^{n-1-i}\sigma_1(i)\sigma_3(j)\tau(n-i-j),
\end{equation*}

\begin{equation*}
F(n) = \tau(n) + 240\sum_{m=1}^{n-1}\sigma_3(m)\tau(n-m),
\end{equation*}
and
\begin{align*}
S =& 80\,L_{E_2,E_2E_4\Delta,n}(17) -60\,L_{E_2,E_4\Delta,n}(16) - 60\,L_{E_2,E_2E_4\Delta,n}(16) +\\& 48\,L_{E_2,E_4\Delta,n}(15) + 960\,L_{E_2^2,E_2E_4\Delta,n}(17) - 720\,L_{E_2^2,E_4\Delta,n}(16).
\end{align*}
From \eqref{l}, we conclude that $S \in \mathbb{Q}.$

\section{Acknowledgement}
The author would like to express deep gratitude to Dr. Brundaban Sahu for many useful discussions and suggestions. Also,  the author is grateful to Dr. Mrityunjoy Charan for suggesting the article \cite{wang} and to  DAE for providing the Institute fellowship.

\end{document}